\documentclass[a4paper,12pt]{article}
\usepackage[english]{babel}
\usepackage[T1]{fontenc}
\usepackage[utf8]{inputenc}
\usepackage{lmodern}

\usepackage{amsthm,amsmath,amsfonts}

\usepackage{bm}
\usepackage{natbib}
\usepackage{enumitem}
\usepackage{mathtools}
\usepackage{prodint}
\usepackage{booktabs}
\usepackage{tikz}

\newcounter{maincounter}
\newtheorem{proposition}[maincounter]{Proposition}
\newtheorem{lemma}[maincounter]{Lemma}

\newcommand{\given}{\mathbin{\mid}}
\newcommand{\R}{\mathbb{R}}
\renewcommand{\d}{\hspace{1pt} \mathrm{d}}
\newcommand{\F}{\mathcal{F}}
\newcommand{\J}{\mathcal{J}}
\renewcommand{\P}{\operatorname{P}}
\newcommand{\E}{\operatorname{E}}
\newcommand{\filter}[1]{\mathcal{#1}}
\newcommand{\sigAlg}[1]{\mathcal{#1}}
\newcommand\independent{\protect\mathpalette{\protect\independenT}{\perp}}
\def\independenT#1#2{\mathrel{\rlap{$#1#2$}\mkern2mu{#1#2}}}
\newcommand{\indic}[1]{\mathbf{1}(#1)}
\newcommand{\mat}[1]{\mathbf{#1}}
\newcommand{\identmat}{\mat{I}}
\newcommand{\from}{:}

\title{On the assumption of independent right censoring}
\author{Morten Overgaard$^1$ and Stefan Nygaard Hansen$^2$}
\date{%
    \small Department of Public Health, Aarhus University \\
    \small Bartholins All\'{e} 2 - Building 1260, DK-8000 Aarhus C, Aarhus, Denmark\\[2ex]
    \small $^1$moov@ph.au.dk\\%
    \small $^2$stefanh@ph.au.dk\\[2ex]%
    May 7, 2019
}

\begin{document}

\maketitle

\begin{abstract}
Various assumptions on a right-censoring mechanism to ensure consistency of the Kaplan--Meier and Aalen--Johansen estimators in a competing risks setting are studied. 
Specifically, eight different assumptions are seen to fall in two categories: a weaker identifiability assumption, which is the weakest possible assumption in a precise sense, and a stronger representativity assumption which ensures the existence of an independent censoring time. 
When a given censoring time is considered, similar assumptions can be made on the censoring time. This allows for a characterization of so-called pointwise independence as well as full independence of censoring time and event time and type.
Examples illustrate how the various assumptions differ.

\vspace{1em}\noindent\textbf{Keywords:} Censoring; competing risks; consistency; identifiability; product integral; representativity.
\end{abstract}

\section{Introduction}

When dealing with right censoring in survival analysis, assumptions on the censoring mechanism are inevitably needed in order to bridge the gap between the observable world and the underlying world of interest. Many seemingly different assumptions have been proposed in the literature. The papers of \cite{williams1977models}, \cite{kalbfleisch1979constant}, and \cite{lagakos1979general} did clarify connections between some of the different assumptions. Since then, martingale theory has become a much used tool in survival analysis, and assumptions on the censoring mechanism are made by means of martingale assumptions in, for instance, \cite{aalen1978empirical}, \cite{Gill1980}, and \cite{Andersen1993}. A clear overview and comparison of the various assumptions does, however, seem to be lacking.

The purpose of this paper is to provide a clear overview and a comparison of various assumptions on the censoring mechanism made in order to ensure consistency of estimators such as the Kaplan--Meier and Aalen--Johansen estimators. This is done in a competing risks setting and without assuming absolute continuity of the involved random variables. Along the way, we obtain assumptions that are minimal in a precise sense for ensuring this consistency. We also make clear that important differences exist between considering a given, underlying censoring time producing right censoring and not considering such a censoring time.

Additionally, the use of product integrals and the techniques used in the many proofs might, in itself, be of interest to researchers in the field of theoretical survival analysis. In particular, the appendix provides a wealth of technical results that may be useful in other settings as well. 

The paper is structured as follows. In Section~\ref{sec:given_event_time}, right censoring in a general form is studied and minimal conditions to ensure consistency of the Kaplan--Meier and Aalen--Johansen estimators are obtained. Various assumptions from the literature are discussed and it is shown that they only correspond to two nested properties: an identifiability assumption and a representativity assumption -- the latter being the strongest. Section~\ref{sec:given_censoring_time} concerns the setting where an explicit censoring time is given. We discuss assumptions on the censoring mechanism and show that independence of the event and censoring times is equivalent to representativity assumptions on both the event and censoring time. In Section~\ref{sec:examples} we treat two examples in order to show that the representativity assumption is strictly stronger than the identifiability assumption and to illustrate the assumptions in a practical setting. Finally, in Section~\ref{sec:discussion} we discuss some of the perspectives of the paper.

\section{A censored event time}
\label{sec:given_event_time}

Consider an event time $T>0$ and event type $D\in\{1,\ldots,d\}$ that are subject to right censoring meaning that we are only able to observe a $\tilde T>0$ with $\tilde T\leq T$ and an indicator $\tilde D=D\indic{\tilde T=T}$ with values in $\{0,\ldots,d\}$ where 0 indicates a censoring.
These are all considered proper random variables, that is, with $\P(\tilde T < \infty) = \P(T < \infty) = 1$.
We will refer to $\tilde T$ and $\tilde D$ as the observed exit time and exit type, respectively, because the risk set is exited at time $\tilde T$ and $\tilde D$ states how. 
This setting does not involve an explicit, underlying censoring time and may be useful in certain practical settings where such a censoring time is difficult to define.
A setting with a given censoring time is dealt with in the next section.

For the pair $(T,D)$ of interest we define the survival function $S(t)=\P(T>t)$, the cause-specific cumulative incidence functions $F_j(t)=\P(T\leq t,D=j)$ for $j=1,\ldots,d$ and the cause-specific cumulative hazard functions $H_j(t)=\int_0^tS(s-)^{-1} F_j(\d s)$ for $j=1,\ldots,d$. We define the corresponding functions for the observed pair $(\tilde T,\tilde D)$, that is, $\tilde S(t)=\P(\tilde T>t)$, $\tilde F_j(t)=\P(\tilde T\leq t,\tilde D=j)$ and $\tilde H_j(t)=\int_0^t\tilde S(s-)^{-1} \tilde F_j(\d s)$ for $j=0,\ldots,d$. Both $H_j$ and $\tilde H_j$ are well-defined functions from $[0,\infty)$ into $[0,\infty]$ for $j=1, \dots, d$. Here and in the following, division by 0 can be interpreted as 0 or any arbitrary number since it only occurs in integrals on a null set of the integrator. 
Frequently, a restriction to the interval $\J = \{t \in [0, \infty) \given \tilde S(t) > 0\}$ is relevant since we will never observe an exit time beyond $\J$. Let $\tau$ denote $\sup\{t > 0 : \tilde S(t) > 0\}$ and note that either $\J = [0, \tau)$, when $\tilde S(\tau-) = 0$, or $\J = [0, \tau]$, when $\tilde S(\tau-) > 0$.

In this section we shall study the assumptions under which we can identify $S$ and $F_j$ by the Kaplan--Meier and Aalen--Johansen estimators defined in Appendix~2. To this end, let $\mat{P}$ denote the $(d+1)\times (d+1)$ matrix of transition probabilities
\begin{align*}
    \mat{P}(s,t)=\begin{pmatrix} S(t\given s) & F_1(t\given s) & \cdots & F_d(t\given s)\\
    0 & 1 & \cdots & 0 \\
    \vdots & \vdots & \ddots & \\
    0 & 0 & \cdots & 1\end{pmatrix}
\end{align*}
where $S(t\given s)= \P(T > t \given T > s) = S(t)/S(s)$ and $F_j(t\given s)= \P(T \leq t, D = j \given T > s) = (F_j(t)-F_j(s))/S(s)$ for $j=1,\ldots,d$ and $t\geq s$. With a slight abuse of notation, we let $\mat{P}(t)=\mat{P}(0,t)$ which is the matrix of interest. If $H(t)=\sum_{j=1}^d H_j(t)$ denotes the all-cause cumulative hazard function and $\tilde H(t)=\sum_{j=1}^d \tilde H_j(t)$ denotes the observed counterpart, then we define the two $(d+1)\times (d+1)$ matrices
\begin{align}
    \label{eq:h_matrix}
    \mat{H}(t)=\begin{pmatrix}-H(t) & H_1(t) & \cdots & H_d(t)\\
    0 & 0 & \cdots & 0 \\
    \vdots & \vdots & \ddots & \vdots \\
    0 & 0 & \cdots & 0\end{pmatrix}, \quad 
    \tilde{\mat{H}}(t)=\begin{pmatrix}-\tilde H(t) & \tilde H_1(t) & \cdots & \tilde H_d(t) \\
    0 & 0 & \cdots & 0 \\
    \vdots & \vdots & \ddots & \vdots\\
    0 & 0 & \cdots & 0\end{pmatrix}
\end{align}
and again, with slight abuse of notation, we let $\mat{H}(s,t)=\mat{H}(t)-\mat{H}(s)$ and $\tilde{\mat{H}}(s,t)=\tilde{\mat{H}}(t)-\tilde{\mat{H}}(s)$ for $t\geq s$.

According to \eqref{eq:limit_aj} of the appendix, the Aalen--Johansen estimator $\hat{\mat{P}}_n(t)$, defined in \eqref{eq:def_aj} of the appendix, is consistent for $\prodi_0^t(\identmat+\tilde{\mat{H}}(\d s))$ for any $t\in\J$ in a setting with independent and identically distributed observations. We now have the following result.

\begin{proposition}
\label{prop:consistency}
In a setting with $n$ independent and identically distributed observations, the Aalen--Johansen estimator $\hat{\mat{P}}_n(t)$ consistently estimates $\mat{P}(t)$ for all $t \in \J$ if and only if $\tilde{\mat{H}}(t)=\mat{H}(t)$ for all $t\in\J$.
In other words, the Aalen--Johansen estimator of $F_j(t)$ is consistent for all $t \in \J$ for $j=1, \dots, d$ if and only if $\tilde H_j(t) = H_j(t)$ for all $t \in \J$ for $j=1, \dots, d$.
\end{proposition}

\begin{proof}
By uniqueness of the product integral, we immediately have $\mat{H}(t)=\tilde{\mat{H}}(t)$ for all $t\in\J$ if and only if $\prodi_0^t(\identmat+\mat{H}(\d s))=\prodi_0^t(\identmat+\tilde{\mat{H}}(\d s))$ for all $t\in \J$. This is due to Theorem~3 of \cite{Gill1990} since both $\mat{H}$ and $\tilde{\mat{H}}$ are seen to be of bounded variation on $[0,t]$ for any $t \in \J$. Now, $\mat{P}$ is seen to satisfy the requirements of Lemma~\ref{lemma:prodi_dH} of the appendix by definition of $H_j$ from which it follows that $\mat{P}(t)=\prodi_0^t (\identmat+\mat{H}(\d s))$. This establishes the equivalence.
\end{proof}

A similar argument reveals that the Kaplan--Meier estimator $\hat S_n(t)$ from \eqref{eq:kaplanmeier} in the appendix consistently estimates $S(t)$ for all $t \in \J$ if and only if $\tilde H(t) = H(t)$ for all $t \in \J$.

We call the property of Proposition~\ref{prop:consistency} the property of \emph{identity of forces of mortality} with inspiration from \cite{elandt1976conditional}.
An assumption of identity of forces of mortality is, for instance, used by \cite{gail1975review} in a competing risks setting as a weaker substitute for the assumption of independent latent event times.

\cite{williams1977models} study, in a setting without competing risks, the constant-sum assumption as a weaker alternative to the assumption of independence of event time and censoring time. 
Let $a_j$ be the function, unique up to $F_j$-null sets, given by $a_j(t) = \P(\tilde T=t,\tilde D=j\given T=t,D=j)$ and let $B(t)=\int_0^{t-}S(s)^{-1} \tilde F_0(\d s)$. In the competing risks setting, the \emph{constant-sum property} can then be phrased as
\begin{equation*}
    a_j(t) + B(t) = 1
\end{equation*}
for $F_j$-almost all $t \in \J$ for $j=1, \dots, d$.
In the paper of \cite{kalbfleisch1979constant}, the authors argue that this property is equivalent to identity of forces of mortality in a setting without competing risks and with a differentiable event hazard function.

Estimators in survival analysis and in the competing risks setting have often been studied using martingale theory, for instance in \cite{aalen1978empirical}, \cite{Gill1980}, \cite{jacobsen1989right}, and \cite{Andersen1993}. In such a setting, the following martingale property, which we will call \emph{the weak martingale property} in light of stronger properties introduced later on, has been shown to ensure the desired consistency of estimators. Let $\tilde N_j(t) = \indic{\tilde T \leq t, \tilde D = j}$ for $j=0, \dots, d$ and $\tilde Y(t) = \indic{\tilde T \geq t}$. The weak martingale property is that the processes given by
\begin{equation*}
    \tilde N_j(t) - \int_0^t \tilde Y(s) H_j(\d s),
\end{equation*}
for $t \geq 0$, for $j=1, \dots, d$ are all martingales with respect to the filtration given by $\tilde \F_t = \sigma(\tilde N_j(s) : j \in \{0, \dots, d\}, s \leq t)$, which models the observed information.
This or similar assumptions are made, for instance, in Assumption~3.1.1 of \cite{Gill1980}, in (2.9) of \cite{jacobsen1989right}, in Definition~3.1.1 of \cite{martinussen2007dynamic}, in (5.5) of \cite{kalbfleisch1980statistical}, and in Theorem~1.3.1 of \cite{fleming1991}.

Recall that $a_j(t) = \P(\tilde T=t,\tilde D=j\given T=t,D=j)$. We consider here yet another property, which we call \emph{status-independent observation}. Status-independent observation is the property that
\begin{equation*}
    a_j(t) = \P(\tilde T \geq t \given T \geq t)
\end{equation*}
for $F_j$-almost all $t \in \J$ for $j=1, \dots, d$, and it is called so because it states that between the statuses of surviving up to a certain time, $T \geq t$, and having some event at that time, $T = t$ with $D =j$, the probability, given a certain status, of that status actually being observed does not depend on the status.

As the following result shows, these four properties are in fact equivalent, and we will refer them collectively as the \emph{identifiability property} in light of Proposition~\ref{prop:consistency}.
\begin{proposition} \label{prop:weak_censoring}
The following properties are equivalent.
\begin{enumerate}[label=\textup{(\ref{prop:weak_censoring}.\arabic*)}]
\item Identity of forces of mortality: $\tilde H_j(t) = H_j(t)$ for $j= 1, \dots, d$ and for any $t \in \J$.  \label{it:Lambda}
\item The weak martingale property: The processes given by $\tilde N_j(t) - \int_0^t \tilde Y(s) H_j(\d s)$, $t \geq 0$, for $j=1, \dots, d$ are all martingales with respect to the filtration $(\tilde \F_t)$, the observed information. \label{it:martingale}
\item Status-independent observation: $a_j(t) = \P(\tilde T \geq t \given T \geq t)$ for $F_j$-almost all $t \in \J$ for $j=1,\ldots,d$. \label{it:cond_prob_given_T}
\item The constant-sum property: $a_j(t)+B(t)=1$ for $F_j$-almost all $t \in \J$ for $j=1, \dots, d$. \label{it:constant_sum}
\end{enumerate} 
\end{proposition}

\begin{proof}
We consider it well known that $\tilde N_j(t) - \int_0^t \tilde Y(s) \tilde H_j(\d s)$, $t \geq 0$ defines a martingale with respect to  $(\tilde \F_t)$. Under the assumption of \ref{it:Lambda} and since $\tilde Y$ is 0 and there is no increment in $\tilde N_j$ outside $\J$ almost surely, we have that
\begin{equation*}
\begin{aligned}
&\phantom{{}=}\tilde N_j(t) - \int_0^t \tilde Y(s) H_j(\d s) = \int_{(0,t]\cap \J} (\tilde N_j(\d s) - \tilde Y(s) H_j(\d s)) \\
&= \int_{(0,t]\cap \J} (\tilde N_j(\d s) - \tilde Y(s) \tilde H_j(\d s)) = \tilde N_j(t) - \int_0^t \tilde Y(s) \tilde H_j(\d s),
\end{aligned}
\end{equation*}
almost surely for all $t \geq 0$ which yields the result.
On the other hand, assume that \ref{it:martingale} holds. Then, for a given $j \in \{1, \dots, d\}$ and a given $t \in \J$,
\begin{equation*}
\tilde F_j(t) = \E(\tilde N_j(t)) =  \int_0^t \E(\tilde Y(s)) H_j(\d s) = \int_0^t \tilde S(s-) H_j(\d s).
\end{equation*}
Since $\tilde S(s-) > 0$ for $s \leq t$, integrating $\tilde S(s-)^{-1}$ with respect to both sides establishes
\begin{equation*}
H_j(t) = \int_0^t \frac{1}{\tilde S(s-)} \tilde F_j(\d s) = \tilde H_j(t)
\end{equation*}
and this yields \ref{it:Lambda}.

Generally, $\tilde F_j(t) = \int_0^t a_j(s) F_j(\d s)$ and $\tilde S(s-) = \P(\tilde T \geq s \given T \geq s) S(s-)$. For $t \in \J$, this establishes 
\begin{equation*}
\begin{aligned}
\tilde H_j(t) &= \int_0^t \frac{a_j(s)}{\P(\tilde T \geq s \given T \geq s)} \frac{1}{S(s-)} F_j(\d s) \\
&= \int_0^t \frac{a_j(s)}{\P(\tilde T \geq s \given T \geq s)} H_j(\d s)
\end{aligned}
\end{equation*}
and thereby the equivalence of \ref{it:Lambda} and \ref{it:cond_prob_given_T}, since $H_j$ and $F_j$ have the same null sets on $\J$.
Assume that \ref{it:Lambda} and \ref{it:cond_prob_given_T} hold. By using equation \eqref{eq:prob_vs_B} of the appendix, it can be seen that $B(t) = \P(\tilde T < t \given T \geq t)$ for all $t \in \J$ under this assumption. Since $\P(\tilde T \geq t \given T \geq t) = a_j(t)$ for $F_j$-almost all $t \in \J$ for $j=1, \dots, d$ under the assumption, we have established $a_j(t)+B(t) = 1$ for $F_j$-almost all $t \in \J$ for $j = 1, \dots, d$, which is \ref{it:constant_sum}.
Assume instead that \ref{it:constant_sum} holds.
Equation~\eqref{eq:B_vs_prob} of the appendix implies that, again, $B(t) = \P(\tilde T < t \given T \geq t)$ for all $t \in \J$. Use of the constant-sum condition again then yields $a_j(t) = \P(\tilde T \geq t \given T \geq t)$ for $F_j$-almost all $t \in \J$ for $j=1, \dots, d$, which is \ref{it:cond_prob_given_T}.
\end{proof}

A somewhat stronger martingale property has, however, also been considered. 
Let $N_j(t) = \indic{T \leq t, D = j}$ for $j=1, \dots, d$ and $Y(t) = \indic{T \geq t}$. Define also a filtration by $\F_t = \sigma(N_j(s) : j \in \{1, \dots, d\}, s \leq t)$ and an enlarged filtration by $\filter{G}_t = \F_t \bigvee \tilde \F_t$. What we call the \emph{strong martingale property} is that the processes given by
\begin{equation*}
    N_j(t) - \int_0^t Y(s) H_j(\d s),
\end{equation*}
for $t \geq 0$, for $j=1, \dots, d$ are all martingales with respect to the enlarged filtration $(\filter{G}_t)$. It seems well-known that the processes are martingales with respect to $(\F_t)$. So, loosely speaking, the property states that enlarging the filtration by $(\tilde \F_t)$ does not add any information relevant for the processes. 
This property has similarities to Definition~III.2.1 of \cite{Andersen1993} of an independent right censoring concept, which also requires the underlying martingale processes to be martingales with respect to an enlarged filtration. Similarly, \cite{aalen1978empirical} also require the underlying martingale processes to be martingales with respect to an enlarged filtration.

The property that
\begin{equation*}
    \P(T \leq t, D = j \given \tilde T > s) = \P(T \leq t, D = j \given T > s)
\end{equation*}
for all $t \geq 0$, $s \in \J$ and $j=1,\ldots,d$ plays a role in Theorem~3.1.1 of \cite{Gill1980}, in condition (G) of \cite{jacobsen1989right}, and also matches the interpretation of independent right censoring given by \cite{Keiding2006}, p.\ 466.
We call this the property of \emph{non-prognostic observation} since it implies that, given survival past time $s$, the extra knowledge that the survival past $s$ is observed, $\tilde T > s$, does not influence the prognosis, that is, the probability of having events at a later point in time.

In \cite{williams1977models}, survival is said to be independent of the conditions producing censoring when a property like
\begin{equation*}
    \P(T \leq t, D = j \given \tilde T = s, \tilde D = 0) = \P(T \leq t, D = j \given T > s)
\end{equation*}
for any $t \geq 0$ and $\tilde H_0$-almost all $s \in \J$ holds for $j=1, \dots, d$. With inspiration from \cite{lagakos1979general}, we will call this property \emph{non-prognostic censoring} because, under assumption of this property, the censoring does not provide any prognostic information about the event time or type other than survival to the censoring time.

The following result shows that these three properties are equivalent and, moreover, that they are equivalent to the existence of an independent censoring time. We will refer to them collectively as the \emph{representativity property} because, looking at \ref{it:representation_general} and \ref{it:representation_cond_fixed}, this property implies that those at risk at a given time, $\tilde T>s$, are representative for those being censored at this time, $\tilde T=s,\tilde D=0$, in terms of the event risks.
\begin{proposition} \label{prop:strong_censoring}
The following properties are equivalent.
\begin{enumerate}[label=\textup{(\ref{prop:strong_censoring}.\arabic*)}]
\item The strong martingale property: The processes that are given by $N_j(t) - \int_0^t Y(s) H_j(\d s)$, $t \geq 0$, for $j=1, \dots, d$, are martingales with respect to the enlarged filtration~$(\filter{G}_t)$. \label{it:martingale_enlarged}
\item Non-prognostic observation: $\P(T \leq t, D = j \given \tilde T > s) = \P(T \leq t, D = j \given T > s)$ for all $t \geq 0$ and $s \in \J$.  \label{it:representation_general} 
\item \label{it:representation_cond_fixed}Non-prognostic censoring: $\P(T \leq t, D = j \given \tilde T = s, \tilde D = 0) = \P(T \leq t, D = j \given T > s)$ for all $t \geq 0$ and $\tilde F_0$-almost all $s \in \J$ and $j=1, \dots, d$.   
\item Existence of an independent censoring time: A censoring time, $C>0$, exists such that $\tilde T = T \wedge C$ and $C \independent (T, D)$. \label{it:C_exists}
\end{enumerate} 
\end{proposition}

\begin{proof}
Assume \ref{it:martingale_enlarged} and let $s \in \J$ and $t > s$ be given. Since $\{\tilde T > s\} \in \filter{G}_s$, we can use the martingale property to obtain 
$\E( \indic{T \in (s,t], D = j, \tilde T > s} ) = \E(\int_s^t \indic{T \geq u, \tilde T > s} H_j(\d u) )$ and divide by $\P(\tilde T > s)$ to get
\begin{equation} \label{eq:martingale_to_general}
\P(T \leq t, D = j \given \tilde T > s) = \int_s^t \P(T \geq u \given \tilde T > s) H_j(\d u).
\end{equation}
The $(d+1) \times (d+1)$ matrix-valued function given by $\mat B(s,t) = \{\beta_{ij}(s,t)\}$,  $\beta_{1,j+1}(s,t) = \P(T \leq t, D = j \given \tilde T > s)$ for $j=1, \dots, d$, $\beta_{1,1}(s,t) = 1- \sum_{j=1}^d \beta_{1,j+1}(s,t) = \P(T > t \given \tilde T > s)$, and $\beta_{i,j}(s,t) = \indic{i=j}$ for $i=2, \dots, d+1$ and $j=1,\ldots,d+1$, is right continuous with left limits in both variables and by~\eqref{eq:martingale_to_general} is seen to satisfy the conditions of Lemma~\ref{lemma:prodi_dH}. Thus, we conclude that $\mat B(s,t) = \prodi_s^t(\identmat+\mat H(\d u))$, which then implies that $\mat B(s,t) = \mat P(s,t)$ since $\mat P(s,t) = \prodi_s^t(\identmat+\mat H(\d u))$ as seen earlier. In particular we have $\P(T \leq t, D = j \given \tilde T > s) = \P(T \leq t, D = j \given T > s)$ for all $t \in [s,\infty) \cap \{u : S(u) > 0\}$ by this argument, and this extends to all $t \geq 0$ since $\{T \notin [s,\infty) \cap \{u : S(u) > 0\} \}$ has probability 0 in either probability measure. We have thereby established \ref{it:representation_general}.

Assuming \ref{it:representation_general}, we may argue the other way to obtain, for $u \leq s \leq t$,
\begin{equation*}
\E(\indic{T \in (s,t], D = j, \tilde T > u}) = \E(\int_s^t \indic{T \geq v, \tilde T > u} H_j(\d v)),
\end{equation*}
which is enough to establish the martingale property of \ref{it:martingale_enlarged} since $\filter{G}_t$ is generated by sets of the type $\{ T > t, \tilde T > s \}$ for $s \leq t$ and $\{ T \leq s, D = j, \tilde T > u \}$ for $s,u \leq t$.

Assume again \ref{it:representation_general}. Then we have the strong martingale property, \ref{it:martingale_enlarged}, which is seen to imply \ref{it:martingale} since integration of the $(\filter{G}_{t})$-predictable process $t \mapsto \indic{\tilde T \geq t}$ with respect to the integrator $t \mapsto N_j(t) - \int_0^t Y(s) H_j(\d s)$ yields $t \mapsto \tilde N_j(t) - \int_0^t \tilde Y(s) H_j(\d s)$, which is then a $(\filter{G}_t)$-martingale and thus also a $(\tilde{\filter{F}}_t)$-martingale. 
In light of Proposition~\ref{prop:weak_censoring} this means that \ref{it:Lambda} holds.
For any given $t \geq 0$ and $j \in \{1, \dots, d\}$, equation~\eqref{eq:general_vs_cond_fixed} of the appendix reveals that $\int_s^{\infty}(\P(T \leq t, D = j \given \tilde T = u, \tilde D = 0) - \P(T \leq t, D = j \given T > u)) \tilde F_0(\d u) = 0$ for any $s \geq 0$ since the integrand is 0 for $u>t$, since we are assuming \ref{it:representation_general}, and since the first two integrals of \eqref{eq:general_vs_cond_fixed} are always zero because for $u\geq 0$ either $\tilde H_j(u) = H_j(u)$ for all $j=1, \dots, d$ or $\tilde S(u-) = 0$.
This establishes \ref{it:representation_cond_fixed}.

If we instead assume \ref{it:representation_cond_fixed}, we obtain \ref{it:constant_sum} from equation~\eqref{eq:cond_fixed_vs_constant_sum} of the appendix and so \ref{it:Lambda} from Proposition~\ref{prop:weak_censoring}. Then equation~\eqref{eq:general_vs_cond_fixed} of the appendix shows that \ref{it:representation_general} holds since, again, for $u \geq 0$ either $\tilde H_j(u) = H_j(u)$ for $j=1, \dots, d$ or $\tilde S(u-) = 0$.

Assume now that \ref{it:representation_general} holds and let us show \ref{it:C_exists}. 
The construction used is the one given in Appendix~3 and is based on the modification $\check H_0$ of $\tilde H_0$ as defined in equation~\eqref{eq:check_H0} below. By construction we have that $\tilde T=T\wedge C$. Furthermore, we see how, for $t \leq s$ with $t \in \J$,
\begin{equation*}
\begin{aligned}
\P(T \leq t, D = j, C > s) &= \P(\tilde T \leq t, \tilde D = j, C > s) \\
&= \Prodi_0^s(1 - \check H_0(\d u)) \int_0^t \Prodi_0^{u-}(1 - \tilde H(\d v)) \tilde H_j(\d u) \\
&= \P(C > s) \P(T \leq t, D = j)
\end{aligned}
\end{equation*}
according to equations \eqref{eq:constr_C_vs_TD} and \eqref{eq:constr_C_dist} of the appendix since also \ref{it:Lambda} holds.
The conclusion, $\P(T \leq t, D = j, C > s) =  \P(C > s) \P(T \leq t, D = j)$ remains valid for $t \leq s$ when $t \in (0, \infty) \backslash \J$ and so $s \in (0, \infty) \backslash \J$ since in this case either $\P(C > s) = 0$ or $\P(T \leq t, D=j) = \P(T \in (0,t]\cap\J, D=j)$ because $\P(C \wedge T \in \J) = 1$.
For $t > s$ with $s \in \J$, we have
\begin{equation*}
\begin{aligned}
\P(T \in(s, t], D = j, C > s) &= \P(T \leq t, D = j \given \tilde T > s) \P(\tilde T > s) \\
&= \P(T \leq t, D = j \given T > s) \P(\tilde T > s) \\
&= \int_s^t \Prodi_s^{u-} ( 1 - H(\d v)) H_j(\d u) \Prodi_0^s (1 - (\tilde H + \tilde H_0)(\d u) \\
&= \int_s^t \Prodi_0^{u-} ( 1 - H(\d v)) H_j(\d u) \Prodi_0^s (1 - \check H_0(\d u)) \\
&= \P(T \in (s,t], D = j) \P(C > s),
\end{aligned}
\end{equation*}
using among other things \ref{it:representation_general} and the product structure of \eqref{eq:productstruc_stilde} below. The conclusion $\P(T \in(s, t], D = j, C > s) = \P(T \in (s,t], D = j) \P(C > s)$ remains valid when $s \in (0, \infty)\backslash \J$ since either side is 0 in this case. Put together, this establishes independence of $C$ and $(T, D)$ and so \ref{it:C_exists}.

Under assumption of \ref{it:C_exists} we have $\P(T \leq t, D=j \given \tilde T > s) = \P(T \leq t, D=j \given T > s, C > s) = \P(T \leq t, D=j \given T > s)$, using the independence, and this is \ref{it:representation_general}. 
\end{proof}

As noted by many authors working under assumption of some version of the representativity property, representativity implies identifiability. As demonstrated by \cite{williams1977models} in their setting, the two properties are not equivalent. This is also the case in our setting.

\begin{proposition}
\label{prop:strong_weak_implication}
The representativity property implies the identifiability property, but the reverse does not hold.
\end{proposition}
\begin{proof}
In the proof of Proposition~\ref{prop:strong_censoring}, the implication has already been established.
Let us here present another argument. Assume \ref{it:C_exists} and choose a censoring time $C$ accordingly such that $C \independent (T,D)$. Then \ref{it:cond_prob_given_T} holds since $a_j(t) = \P(C \geq t) = \P(\tilde T \geq t \given T \geq t)$ for $F_j$ almost all $t \in \J$ for $j=1, \dots, d$.
This shows the implication.

On the other hand, the event time $T_2$ and the observed pair $(\tilde T,\tilde D)$ constructed in Section~\ref{sec:examples} below provides an example where identifiability holds but representativity does not.
\end{proof}

\section{Censoring by a given censoring time}
\label{sec:given_censoring_time}

In this section we consider as given an event time $T$, an event type $D$, and a censoring time $C$. The observed pair is thus explicitly $\tilde T = T \wedge C$ and $\tilde D = D \indic{T \leq C}$, which is a special case of the setting in Section~\ref{sec:given_event_time}.

For the censoring time, we denote its survival function $K(t)=\P(C>t)$, distribution function $G(t)=\P(C\leq t)$ and cumulative hazard function $H_0(t)=\int_0^t K(s-)^{-1}G(\d s)$. 

As a result of the asymmetry between $T$ and $C$ in the definition of $(\tilde T, \tilde D)$ where $T$ takes priority, a modification of $\tilde H_0$ is relevant for it to be comparable to the defined $H_0$. We let 
\begin{equation} \label{eq:check_H0}
    \check H_0(t) = \int_0^t \frac{1}{1 - \Delta \tilde H(s)} \tilde H_0(\d s)
\end{equation}
define this modification and note that $\Delta \check H_0(t) (1-\Delta \tilde H(t)) = \Delta \tilde H_0(t)$ and so $(1-\Delta \check H_0(t))(1-\Delta \tilde H(t)) = 1-\Delta (\tilde H_0 + \tilde H)(t)$. The continuous parts of $\check H_0$ and $\tilde H_0$ are the same so by the characterization of the product integral $\tilde S(t) = \prodi_0^t(1-(\tilde H_0 + \tilde H)(\d s))$ of Definition~4 from \cite{Gill1990}, the modification allows for the product structure 
\begin{align}
    \label{eq:productstruc_stilde}
    \tilde S(t) = \Prodi_0^t(1-\tilde H(\d s)) \Prodi_0^t(1- \check H_0(\d s))
\end{align}
which has technical importance in the following.

If we define $\check S(t) = \tilde S(t-)(1- \Delta \tilde H(t)) = \prodi_0^t(1-\tilde H(\d s)) \prodi_0^{t-}(1- \check H_0(\d s))$, this modification can also be expressed as $\check H_0(t) = \int_0^t \check S(s)^{-1} \tilde F_0(\d s)$ using the definition of $\tilde H_0$.
The difference $\check S(t) - \tilde S(t)$ is seen to be $\check S(t) \Delta \check H_0(t) = \Delta \tilde F_0(t) = \P(\tilde T = t, \tilde D = 0)$, and, by letting $\check Y(t) = \indic{\tilde T > t} + \indic{\tilde T = t, \tilde D = 0}$, we see that $\check S(t)=\E(\check Y(t))=\P(T>t,C\geq t)$. 

We now consider properties relating to the censoring similar to those of Proposition~\ref{prop:weak_censoring} which are then naturally termed the \emph{censoring identifiability property}.
\begin{proposition}
\label{prop:weak_eventtime}
The following properties are equivalent.
\begin{enumerate}[label=\textup{(\ref{prop:weak_eventtime}.\arabic*)}]
\item We have that $\check H_0(t) = H_0(t)$ for any $t \in \J$. \label{it:Lambda2}
\item The process given by $\tilde N_0(t)-\int_0^t \check Y(s) H_0(\d s)$, $t\geq 0$, is a martingale with respect to the filtration $(\tilde \F_t)$, the observed information.\label{it:martingale2}
\item We have that $\P(\tilde T=t,\tilde D=0\given C=t)=\P(T > t  \given C \geq t)$ for $G$-almost all $t\in\J$.\label{it:cond_prob_given_C}
\item We have that $\P(\tilde T=t,\tilde D=0\given C=t) + \sum_{j=1}^d\int_0^{t} K(s-)^{-1} \tilde F_j(\d s) = 1$ for $G$-almost all $t \in \J$. \label{it:constant_sum2}
\end{enumerate}
\end{proposition}

\begin{proof}
The equivalence of \ref{it:Lambda2} and \ref{it:martingale2} follows by a similar argument as in the proof of Proposition~\ref{prop:weak_censoring} but now using the fact that $\tilde N_0(t)-\int_0^t\check Y(s)\check H_0(\d s)$, $t\geq 0$, can be shown to define a martingale.

The equivalence of \ref{it:Lambda2} and \ref{it:cond_prob_given_C} is obtained by mimicking the steps in Proposition~\ref{prop:weak_censoring} while using the identity
\begin{align*}
\check H_0(t)=\int_0^t\frac{\P(\tilde T=s,\tilde D=0\given C=s)}{\check S(s)/K(s-)}H_0(\d s).
\end{align*}
The equivalence then follows by noting that $\check S(s)/K(s-)=\P(T > t  \given C \geq t)$.

The identity in \eqref{eq:prob_vs_B2} of the appendix immediately shows that \ref{it:Lambda2} implies \ref{it:constant_sum2} by exploiting the fact that we have already established the equivalence between \ref{it:cond_prob_given_C} and \ref{it:Lambda2}. Similarly, the identity in \eqref{eq:B_vs_prob2} of the appendix immediately shows that \ref{it:constant_sum2} implies \ref{it:cond_prob_given_C}.
\end{proof}

\cite{williams1977models} considered, with inspiration from \cite{gail1975review}, an independent censoring assumption which, in this setting, may be formulated as the following property. The property is
\begin{equation*}
    \tilde F_j(t) = \int_0^t K(s-) F_j(\d s)
\end{equation*}
for $j=1, \dots, d$ and
\begin{equation*}
    \tilde F_0(t) = \int_0^t S(s) G(\d s)
\end{equation*}
for all $t \in \J$. In \cite{williams1977models}, this assumption was seen to be a stronger assumption than the constant-sum property, here given in \ref{it:constant_sum}. As was also noted by \cite{kalbfleisch1979constant}, in the setting of their paper, this is the case only because it includes an additional requirement on the given censoring time. This is the content of the following result.

\begin{proposition} \label{prop:2weak}
The following properties are equivalent.
\begin{enumerate}[label=\textup{(\ref{prop:2weak}.\arabic*)}]
    \item $\tilde H_j(t) = H_j(t)$ for $j=1, \dots, d$ and $\check H_0(t) = H_0(t)$ for all $t \in \J$. \label{it:all_forces_of_mortality}
    \item $\tilde F_j(t) = \int_0^t K(s-) F_j(\d s)$ for $j=1, \dots, d$ and $\tilde F_0(t) = \int_0^t S(s) G(\d s)$ for all $t \in \J$. \label{it:Williams_Lagakos_S2}
    \item $\P(C \geq t \given T = t, D = j) = \P(C \geq t)$ for $F_j$-almost all $t \in \J$ for $j=1, \dots, d$ and $\P(T > t \given C = t) = \P(T > t)$ for $G$-almost all $t \in \J$. \label{it:pointwise_independence}
\end{enumerate}
\end{proposition}

\begin{proof}
Assume that \ref{it:all_forces_of_mortality} holds. The product structure $\tilde S(t) = \prodi_0^t(1-\tilde H(\d s)) \prodi_0^t(1-\check H_0(\d s))$ results in $\tilde S(t) = S(t) K(t)$ and similarly $\check S(t) = S(t) K(t-)$ under the assumption. Using the assumption again, we have $\tilde F_j(t) = \int_0^t \tilde S(s-) H_j(\d s) = \int_0^t K(s-) F_j(\d s)$ and $\tilde F_0(t) = \int_0^t \check S(s) H_0(\d s) = \int_0^t S(s) G(\d s)$ which is \ref{it:Williams_Lagakos_S2}.

Assume now that \ref{it:Williams_Lagakos_S2} holds. Then $B(t) = \int_0^{t-} S(s)^{-1} \tilde F_0(\d s) = G(t-) = 1-K(t-)$ using the last part of the assumption. Using this in combination with the first part of the assumption yields $\tilde F_j(t) = \int_0^t(1-B(s)) F_j(\d s)$. We already know that $\tilde F_j(t) = \int_0^t a_j(s) F_j(\d s)$, so the constant sum property of \ref{it:constant_sum} and hence also \ref{it:Lambda} follow. The property~\ref{it:constant_sum2} and so \ref{it:Lambda2} can be obtained in a similar manner. This establishes \ref{it:all_forces_of_mortality}.

From the equalities $\tilde F_j(t) = \int_0^t \P(C \geq s \given T = s, D = j) F_j(\d s)$ and $\tilde F_0(t) = \int_0^t \P(T > s \given C = s) G(\d s)$ which hold for all $t \in \J$, the properties of \ref{it:Williams_Lagakos_S2} and \ref{it:pointwise_independence} are seen to be equivalent.
\end{proof}

The equivalence between \ref{it:all_forces_of_mortality} and \ref{it:Williams_Lagakos_S2} shows that the property introduced by \cite{williams1977models} is equivalent to having both the identifiability and the censoring identifiability property. The property of \ref{it:pointwise_independence} can be considered \emph{pointwise} independence between $(T,D)$ and $C$ and, as is evident from the proof of Proposition~\ref{prop:2weak}, it also implies $\P(T > t, C > t) = \P(T > t) \P(C > t)$ for all $t \geq 0$. For this reason, we refer to the properties in Proposition~\ref{prop:2weak} collectively as the property of \emph{pointwise independence}. It does not imply independence of $(T,D)$ and $C$, however. 

Independence of $(T, D)$ and $C$ is here referred to as \emph{full independence}. This assumption is made by many authors, and is, for instance, used in \cite{kaplan1958nonparametric}. 
In \cite{lagakos1979general}, the property is described as strictly stronger than the non-prognostic censoring property from Proposition~\ref{prop:strong_censoring}. The next result shows that this is the case only because full independence includes a further property of representativity of the given censoring time. This property is that
\begin{equation} \label{eq:representation_cond_fixed2}
    \P(C \leq t \given \tilde T = s, \tilde D = j) = \P(C \leq t \given C \geq s)
\end{equation}
holds for any $t \geq 0$ and $\tilde F_j$-almost all $s \in \J$ for $j=1, \dots, d$. We will refer to this as the \emph{censoring representativity property} as it is a counterpart to \ref{it:representation_general}. An argument similar to the one used in Proposition~\ref{prop:strong_weak_implication} shows that censoring representativity implies censoring identifiability but that the two properties are not equivalent. The following result now applies.

\begin{proposition} \label{prop:random_censoring}
Full independence, $C \independent (T,D)$, holds if and only if both the representativity property and the censoring representativity property hold.
\end{proposition}

\begin{proof}
It is evident that full independence implies \ref{it:C_exists}. Similarly, $\P(C \leq t \given \tilde T = s, \tilde D = j) = \P(C \leq t \given C \geq s, T = s, D = j) = \P(C \leq t \given C \geq s)$ for any $t \geq 0$ and $\tilde H_j$-almost all $s \in \J$ for $j=1, \dots, d$ under the independence assumption.

Assume instead that the properties of Proposition~\ref{prop:strong_censoring} and \eqref{eq:representation_cond_fixed2} hold. 
By equation~\eqref{eq:cond_fixed_vs_constant_sum2} of the appendix, \ref{it:constant_sum2} and so, by Proposition~\ref{prop:weak_eventtime}, also \ref{it:Lambda2}. Now \eqref{eq:representation_cond_fixed2} states that $\P(C > t \given \tilde T = s, \tilde D = j) = \prodi_{s-}^t(1- \check H_0(\d u))$. 
This is exactly the conditional distribution of the independent censoring time constructed in the proof of Proposition~\ref{prop:strong_censoring}. Since we are assuming that the properties of Proposition~\ref{prop:strong_censoring} hold, the same calculations lead to the independence of $C$ and $(T,D)$.
\end{proof}

\section{Examples}
\label{sec:examples}

\subsection{A technical setting}

This technical example serves to illustrate the differences between the identifiability and representativity properties. 

Consider the probability space $(\Omega, \F, \P)$ with $\Omega = [0,1]^2 = \{ (t,c) \in \R^2 \given 0 \leq t \leq 1, 0 \leq c \leq 1\}$, $\F$ the Borel $\sigma$-algebra, and $\P$ the uniform distribution such that $\P([s,t] \times [u,v]) = (t-s)(v-u)$ for $s \leq t$, $u \leq v$, all in $[0,1]$.
The random variables given by $T_1(t,c) = t$ and $C_1(t,c) = c$ are then independent. We further define the random variables
\begin{align*}
    T_2(t,c) &= \begin{cases} (1-c) & \textup{if } \frac{1}{2} \leq t \leq 1, 0 \leq c < \frac{1}{2} \\  t & \textup{otherwise} \end{cases} \\
    C_2(t,c) &= \begin{cases} (1-t) & \textup{if } 0 \leq t \leq \frac{1}{2}, \frac{1}{2} \leq c \leq 1 \\  c & \textup{otherwise} \end{cases},
\end{align*}
and $C_3(t,c) = c\indic{c < t} + \indic{c \geq t}$. A direct calculation reveals that the distributions of $T_1,T_2,C_1,C_2$ are all uniform on $[0,1]$.

If we define $\tilde T(t,c)=t\wedge c$ and $\tilde D(t,c)=\indic{t\leq c}$, then $\tilde T=T_i\wedge C_j$ and $\tilde D=\indic{T_i\leq C_j}$ for any choice of $i \in \{1,2\}$ and $j \in \{1,2,3\}$. That is, any combination of the event and censoring times defined above yields the same observable exit time and exit type.

Note that the representativity property holds for $T_1$ by virtue of \ref{it:C_exists} because $T_1$ is independent of $C_1$ and $\tilde T=T_1 \wedge C_1$ and so by Proposition~\ref{prop:strong_weak_implication}, the identifiability property also holds for $T_1$. Thus, the identifiability property also holds for $T_2$ since, for instance, the property of identity of forces of mortality is inherited from $T_1$ as $T_1$ and $T_2$ have the same distribution. A calculation reveals that for $\tilde F_0$-almost all $s\in[0,\frac{1}{2})$ we have $\P(T_2 \leq 1-s \given \tilde T = s, \tilde D = 0) = 1$ and $\P(T_2 \leq 1-s \given T_2 > s) = 1 - s/(1-s)$ such that non-prognostic censoring and thereby representativity cannot hold for $T_2$.
Similarly, censoring representativity holds for $C_1$, but cannot hold for $C_2$.

Since, for $t \in [0,t)$, $\P(C_3 \leq t) = \P(\tilde T \leq t, \tilde D = 0) = \tilde F_0(t)$, the cumulative hazard associated with the distribution of $C_3$ is $\int_0^t (1- \tilde F_0(s))^{-1} \tilde F_0(\d s) < \int_0^t (1- \tilde F_0(s) - \tilde F_1(s))^{-1} \tilde F_0(\d s) = \int_0^t \tilde S(s)^{-1} \tilde F_0(\d s) = \tilde H_0(t)$ for all $t \in [0,1)$ such that censoring identifiability cannot hold for $C_3$.

Figure~\ref{fig:distributions} illustrates the definition of $T_i$ and $C_j$ as well as the observed exit time $\tilde T$ as a heat map. Note how, for any combination of $T_i$ and $C_j$, the minimum of their respective graphs correspond to the graph of $\tilde T$. The assumptions met for the various choices of $T_i$ and $C_j$ to produce $(\tilde T, \tilde D)$ are summarized by Table~\ref{tab:example_table}.

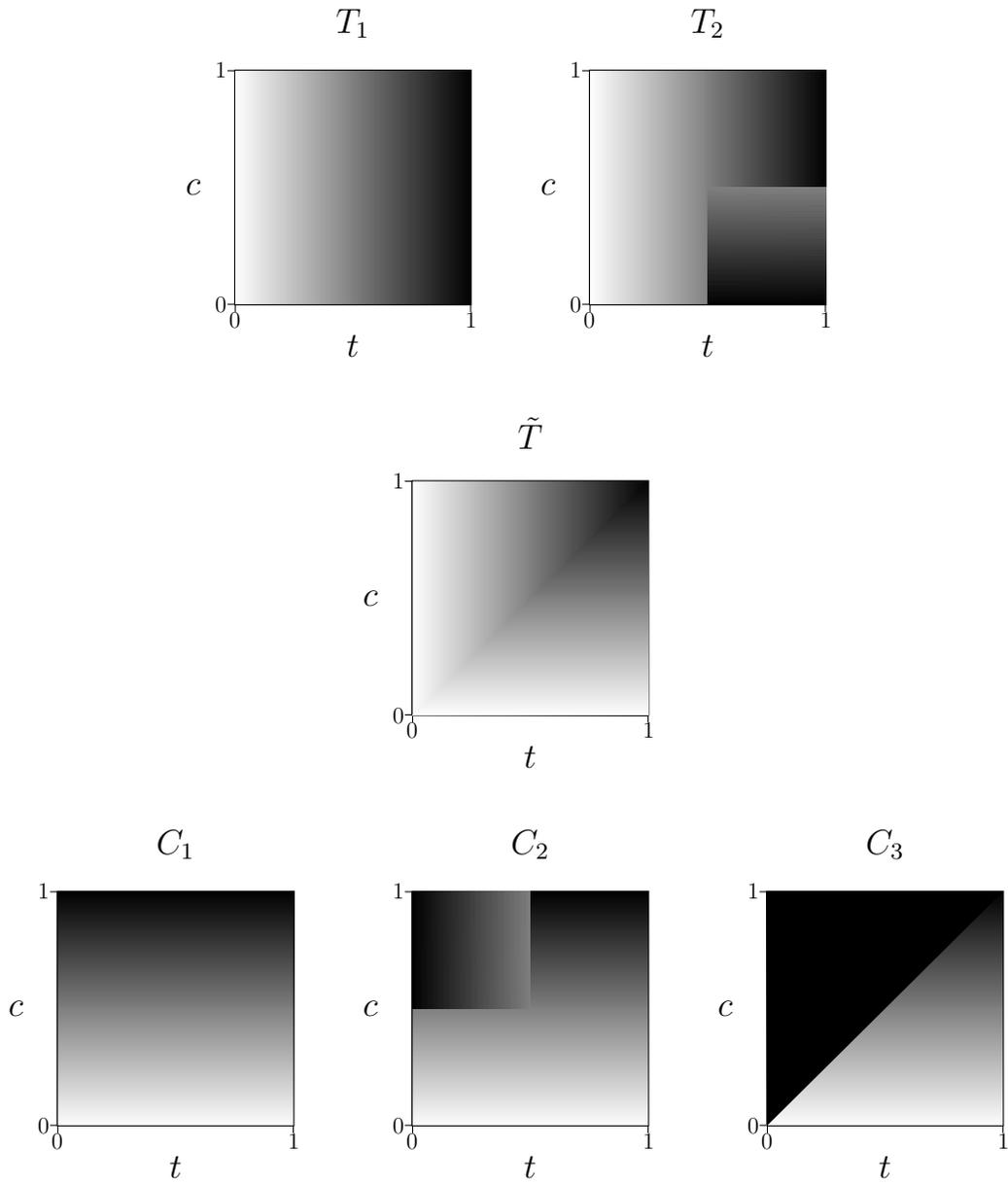
\begin{figure}
    \centering

    \begin{tikzpicture}[scale=3.2, every node/.style={scale=0.75}]
        \draw [shift={(0,0)}] (0.5,1.2) node [scale=1.5] {$T_1$};
        \draw [shift={(0,0)},|-|] (0,0) node [below] {0} -- (1,0) node [midway,yshift=-0.75cm,scale=1.5] {$t$} node [below] {1};
        \draw [shift={(0,0)},|-|] (0,0) node [left]  {0} -- (0,1) node [midway,xshift=-0.75cm,scale=1.5]  {$c$} node [left]  {1};
        \draw [shift={(0,0)},left color=white, right color=black,  ] (0,0) rectangle (1,1);
    
        \draw [shift={(1.5,0)}] (0.5,1.2) node [scale=1.5] {$T_2$};
        \draw [shift={(1.5,0)},|-|] (0,0) node [below] {0} -- (1,0) node [midway,yshift=-0.75cm,scale=1.5] {$t$} node [below] {1};
        \draw [shift={(1.5,0)},|-|] (0,0) node [left]  {0} -- (0,1) node [midway,xshift=-0.75cm,scale=1.5]  {$c$} node [left]  {1};
        \draw [shift={(1.5,0)},left color=white, right color=black,  ] (0,0) rectangle (1,1);
        \shade [shift={(1.5,0)},bottom color=black, top color=gray,  ] (0.5,0) rectangle (1,0.5);
    
        \draw [shift={(-0.75,-3.5)}](0.5,1.2) node [scale=1.5] {$C_1$};
        \draw [shift={(-0.75,-3.5)},|-|] (0,0) node [below] {0} -- (1,0) node [midway,yshift=-0.75cm,scale=1.5] {$t$} node [below] {1};
        \draw [shift={(-0.75,-3.5)},|-|] (0,0) node [left]  {0} -- (0,1) node [midway,xshift=-0.75cm,scale=1.5]  {$c$} node [left]  {1};
        \draw [shift={(-0.75,-3.5)},bottom color=white, top color=black,  ] (0,0) rectangle (1,1);
        
        \draw [shift={(0.75,-3.5)}](0.5,1.2) node [scale=1.5] {$C_2$};
        \draw [shift={(0.75,-3.5)},|-|] (0,0) node [below] {0} -- (1,0) node [midway,yshift=-0.75cm,scale=1.5] {$t$} node [below] {1};
        \draw [shift={(0.75,-3.5)},|-|] (0,0) node [left]  {0} -- (0,1) node [midway,xshift=-0.75cm,scale=1.5]  {$c$} node [left]  {1};
        \draw [shift={(0.75,-3.5)},bottom color=white, top color=black,  ] (0,0) rectangle (1,1);
        \shade [shift={(0.75,-3.5)},left color=black, right color=gray,  ] (0,0.5) rectangle (0.5,1);
        
        \draw [shift={(2.25,-3.5)}](0.5,1.2) node [scale=1.5] {$C_3$};
        \draw [shift={(2.25,-3.5)},|-|] (0,0) node [below] {0} -- (1,0) node [midway,yshift=-0.75cm,scale=1.5] {$t$} node [below] {1};
        \draw [shift={(2.25,-3.5)},|-|] (0,0) node [left]  {0} -- (0,1) node [midway,xshift=-0.75cm,scale=1.5]  {$c$} node [left]  {1};
        \draw [shift={(2.25,-3.5)},bottom color=white, top color=black,  ] (0,0) rectangle (1,1);
        \fill [shift={(2.25,-3.5)},fill=black] (0,0) -- (1,1) -- (0,1);
    
        \draw [shift={(0.75,-1.75)}] (0.5,1.2) node [scale=1.5] {$\tilde T$};
        \draw [shift={(0.75,-1.75)},|-|] (0,0) node [below] {0} -- (1,0) node [midway,yshift=-0.75cm,scale=1.5] {$t$} node [below] {1};
        \draw [shift={(0.75,-1.75)},|-|] (0,0) node [left]  {0} -- (0,1) node [midway,xshift=-0.75cm,scale=1.5]  {$c$} node [left]  {1};
        \draw [shift={(0.75,-1.75)},left color=white, right color=black,  ] (0,0) rectangle (1,1);
        \shade [shift={(0.75,-1.75)},bottom color=white, top color=black,  ] (0,0) -- (1,1) -- (1,0);
    
    \end{tikzpicture}
    
    \caption{Illustration of the definitions of the random variables of the technical example. Here, white is 0, black is 1, and gray is in between with darker meaning closer to 1.}
    \label{fig:distributions}
\end{figure}

\begin{table}[htbp]
    \centering
    \caption{Assumptions met for various combinations of $T_i$ and $C_j$.}
    \makebox[\linewidth]{
    \begin{tabular}{lcccccc} \addlinespace[12pt] \toprule
         Assumption & $(T_1,C_1)$ & $(T_1,C_2)$ & $(T_1,C_3)$ & $(T_2,C_1)$ & $(T_2,C_2)$ & $(T_2,C_3)$ \\ \midrule
        Identifiability          & \checkmark & \checkmark & \checkmark & \checkmark  & \checkmark & \checkmark   \\
        Representativity         & \checkmark & \checkmark & \checkmark &             &            &              \\
        Cens.\ identifiability   & \checkmark & \checkmark &            & \checkmark  & \checkmark &              \\
        Cens.\ representativity  & \checkmark &            &            & \checkmark  &            &              \\\midrule
        Pointwise independence   & \checkmark & \checkmark &            & \checkmark  & \checkmark &              \\
        Full independence        & \checkmark &            &            &             &            &              \\\bottomrule
    \end{tabular}
    \label{tab:example_table}
    }
\end{table}

The primary idea behind these examples is that with basis in independent $T_1$ and $C_1$, we can alter the unobserved parts of the underlying event and censoring time without altering the observed $(\tilde T, \tilde D)$. If the event time is left unaltered but the unobserved part of the censoring time is altered arbitrarily, the representativity property is retained. If the marginal distribution is retained as is the case for $T_2$ and $C_2$, the identifiability property is retained.

\subsection{A practical setting}

As an illustration of a practical setting, we can consider the following example of a register-based study. 
Suppose we are interested in studying the cumulative incidences of different causes of death in a certain population. In this case, we can let $(T,D)$ denote the pair of time of death and cause of death for a randomly picked member of the population.
Imagine that we have information on age and cause of death of population members except in the case of emigration from the population. In other words, we have information on $\tilde T \leq T$, which equals $T$ if the time of death is observed and is the time of emigration otherwise, and $\tilde D = D \indic{\tilde T = T}$, which is the cause of death if the time of death is observed and 0, denoting emigration, otherwise.
As can be seen from Proposition~\ref{prop:existence} of the appendix, data on the observed pair $(\tilde T, \tilde D)$ alone does not allow us to refute the idea that $(\tilde T, \tilde D)$ is produced by $(T,D)$ and a time to emigration $C$ that are independent, $C \independent (T,D)$.
However, in this case, common sense tells us that emigration cannot happen after death so the time to emigration $C$ can never be independent of $(T,D)$. Instead, one should rather define $C= \infty$ when $\tilde D \neq 0$ or, as in Section~\ref{sec:given_event_time}, simply not trouble oneself with defining a time to emigration for all individuals.

Suppose now we are interested in estimating the cumulative incidence proportion $F_j(t)$ for various time points $t \in \J$ for the different causes of death $j=1, \dots, d$. 
In the end, the problem of defining a time to emigration has no bearing on the validity of the Aalen--Johansen estimator as an estimate of $F_j(t)$. We instead require the identifiability property relating $(T, D)$ to $(\tilde T, \tilde D)$ directly as laid out in Proposition~\ref{prop:weak_censoring}.
In terms of the identity of forces of mortality property, this requirement has the interpretation that the observable hazard of any of the causes of death as represented by $\tilde H_j$ should equal the underlying hazard of the same cause as represented by $H_j$ on the relevant time interval.
In terms of the status-independent observation property, the requirement has the interpretation that the status of survival up to any time point and the status of death of a certain cause at the same time point are equally likely to be observed, that is, the probability of not emigrating before that time point given the status does not depend on the status.

If we are instead interested in using the Aalen--Johansen estimator for prognosis, we need a stronger assumption. Suppose we are looking at population members that are alive and have not emigrated at time point $s$ and we are interested in estimating the probabilities of dying of the different causes before time $t > s$. In other words, we are interested in estimating $\P(T \leq t, D = j \given \tilde T > s)$. Under the identifiability assumption, a valid estimate of $\P(T \leq t, D = j \given T > s) = (F_j(t) - F_j(s))/S(s)$ can be obtained based on the Aalen--Johansen and related Kaplan--Meier estimator.
In order for this estimate to be a valid estimate of $\P(T \leq t, D = j \given \tilde T > s)$ as well, an assumption of the non-prognostic observation property from Proposition~\ref{prop:strong_censoring} is needed. That is, we require the stronger representativity property to hold. By the equivalence to non-prognostic censoring, this entails that population members emigrating at time $s$ have the same probabilities of dying of certain causes as members that are alive at time $s$. 

The representativity assumption also implies the existence of a censoring time $C$ independent of $(T,D)$ which corresponds to the time to emigration for individuals who emigrate. It may be useful to think in terms of such a $C$, but its value for individuals who are not observed to be emigrating should not, at least without further assumptions, be confused with a counterfactual emigration time that would have been observed if death had not occurred beforehand. In fact, the censoring time $C$ may not have any relevant interpretation for individuals who are observed to die.

In register-based studies, censoring at end of follow up may be much more prominent than, for instance, censoring by emigration from the population. End of follow-up is an example of a censoring time that may defined explicitly without consideration of the underlying $(T,D)$. This extra piece of information can be used to judge whether censoring identifiability and censoring representativity are appropriate, but these properties do not help us in judging the validity of the representativity or identifiability properties related to $(T,D)$ and thus the validity of the Aalen--Johansen and Kaplan--Meier estimators.

\section{Discussion}
\label{sec:discussion}

When no given, underlying censoring time is considered, the assumptions that we have studied that ensure consistency of the Kaplan--Meier and Aalen--Johansen estimators fall in two categories: an identifiability assumption and a representativity assumption. Although the properties within one category are all equivalent, they are quite different in their interpretations and hence some may be easier to communicate to a clinical researcher than others. Which interpretation is most suitable is a matter of preference but it seems to us that the properties of status-independent observation and non-prognostic observation are much easier to interpret and potentially refute than, for example, the corresponding martingale properties.

The appropriateness of either assumption cannot be assessed based on information on the exit time and exit type alone as is seen from Proposition~\ref{prop:existence} of the appendix, which ensures the existence of an event time and type that realize the observed exit time and exit type and at the same time satisfy the representativity assumption. This is in a similar vein to the result by \cite{molenberghs2008} that one cannot distinguish between missing-at-random and missing-not-at-random models based on only the observed data.
Consequently, any information used to the assess the validity of the identifiability or representativity assumption must come from an external source. 

Other properties than the ones treated in this paper have been considered in the literature. \cite{ebrahimi2003} considered a certain property and proceeded to argue its equivalence to the constant-sum property.
\cite{jacobsen1989right} considered three nested properties in a setting where the observed censoring times need not be independent and identically distributed, see his Proposition~3.4. 

It appears that these three properties all translate into equivalents of the representativity property in our setting.

Our focus has been on marginal distributions, but in regression analysis in a survival analysis context, a similar question of necessary assumptions on the censoring mechanism is highly relevant. 
Seemingly, versions of the assumptions studied here in the conditional distribution given covariates of a regression model are useful in this respect.

\section*{Acknowledgements}
The authors would like to thank an anonymous referee for valuable comments that have greatly improved the manuscript.
Morten Overgaard is supported by the Novo Nordisk Foundation grant NNF17OC0028276.

\appendix

\section*{Appendix 1}
\subsection*{Technical results}

Consider the matrix $\mat{H}$ defined in \eqref{eq:h_matrix}. We then have the following characterization of the product integral $\prodi (\mat{I}+\d\mat{H})$.

\begin{lemma} \label{lemma:prodi_dH}
Consider a $(d+1) \times (d+1)$ matrix-valued function given by $\mat B(s,t) = \{\beta_{i,j}(s,t)\}$ for $s, t \geq 0$ which is right continuous with left limits in both variables. Then, for given $s \geq 0$, $\mat B(s,t) = \prodi_s^t(\identmat+\mat H(\d u))$ for all $t \in [s, \infty) \cap \{t : S(t) > 0\}$ if and only if
\begin{equation*}
    \beta_{1,j+1}(s,t) = \int_s^t \beta_{1,1}(s,u-) \d H_j(u)
\end{equation*}
for $j=1, \dots, d$, $\beta_{1,1}(s,t) = 1 - \sum_{j=1}^d \beta_{1,j+1}(s,t)$, and $\beta_{i,j}(s,t) = \indic{i=j}$ for $i = 2, \dots, d+1$ for all $t \in [s, \infty) \cap \{t : S(t) > 0\}$.
\end{lemma}
\begin{proof}
This is a special case of Theorem~5 of \cite{Gill1990}, which establishes that $\mat B(s,t) = \prodi_s^t(\identmat+\mat H(\d u))$ if and only if the forward equation $\mat B(s,t) - \identmat = \int_s^t \mat B(s,u-) \mat H(\d u)$ for all $t \in [s, \infty) \cap \{t : S(t) > 0\}$ holds. The only solutions to the equations $\beta_{i,1}(s,t) = - \int_s^t \beta_{i,1}(s, u-)  H(\d u)$ for all $t \in [s, \infty) \cap \{t : S(t) > 0\}$ for $i=2, \dots, d+1$ implied by the forward equation, are $\beta_{i,1}(s,t) = 0$, see for instance Theorem~10 of \cite{Gill1990}. This, in turn, implies that $\beta_{i,j}(s,t) = \indic{i=j}$ for $i=2, \dots, d+1$ for $\mat B$ to be a solution to the forward equation.
\end{proof}

In the following, we give some useful identities in the setup of Section~\ref{sec:given_censoring_time} where event time $T$, event type $D$, and censoring time $C$ are all given and we observe $\tilde T=T\wedge C$ and $\tilde D=D\indic{T\leq C}$. The identities not involving the $C$ may, however, also be used in the setting of Section~\ref{sec:given_event_time} where a censoring time $C$ is not explicitly given.
\begin{lemma}
The equalities
\begin{equation} \label{eq:prob_vs_B}
    \P(\tilde T < t \given T \geq t) = B(t) + \int_0^{t-} \frac{\tilde S(s-)}{S(s)} (\tilde H - H)(\d s)
\end{equation}
and
\begin{equation}
    \label{eq:prob_vs_B2}
    \P(T \leq t  \given C \geq t) = \sum_{j=1}^d \int_0^t \frac{1}{K(s-)} \tilde F_j(\d s) + \int_0^{t-} \frac{\check S(s)}{K(s)} (\check H_0 - H_0)(\d s)
\end{equation}
hold for all $t \in \J$.
\end{lemma}
\begin{proof}
Since the product integral structures $S(t)/S(s) = \prodi_s^t(1-H(\d u))$ and $\tilde S(t)/ \tilde S(s) = \prodi_s^t(1- (\tilde H_0 + \tilde H)(\d u))$ hold, the equality $S(t) - \tilde S(t) = \int_0^t S(t) S(s)^{-1} \tilde S(s-) (\tilde H_0 + \tilde H - H)(\d s)$ holds according to the Duhamel equation, see Theorem~6 of \cite{Gill1990}. 
Note that $S(t-) - \tilde S(t-) = \P(\tilde T < t, T \geq t)$ and recall that $B(t) = \int_0^{t-}S(s)^{-1} \tilde F_0(\d s) = \int_0^{t-}\tilde S(s-) S(s)^{-1} \tilde H_0(\d s)$ to obtain the equality~\eqref{eq:prob_vs_B}. 

A similar argument leads to $\P(T \leq t  \given C \geq t) = \sum_{j=1}^d \Delta F_j(s) / K(t-) + \sum_{j=1}^d \int_0^{t-} K(s)^{-1} \tilde F_j(\d s) + \int_0^{t-} \tilde S(s-) K(s)^{-1} (\tilde H_0 - H_0)(\d s)$. The equality~\eqref{eq:prob_vs_B2} now follows by realizing that $\int_0^{t-} \check S(s) K(s)^{-1} \check H_0(\d s) = \int_0^{t-} \tilde S(s-) K(s)^{-1} \tilde H_0(\d s)$ and $\int_0^{t-} (\tilde S(s-)- \check S(s)) K(s)^{-1} H_0(\d s) = \sum_{j=1}^d \sum_{s < t} \Delta \tilde F_j(s) \Delta H_0(s) / K(s)  = \sum_{j=1}^d \int_0^{t-} (K(s-)^{-1}-K(s)^{-1}) \tilde F_j(\d s)$.
\end{proof}

\begin{lemma}
The equalities
\begin{equation} \label{eq:B_vs_prob}
    B(t) = \P(\tilde T < t \given  T \geq t) + \frac{1}{S(t-)} \sum_{j=1}^d \int_0^{t-} (1 - a_j(s) - B(s)) F_j(\d s)
\end{equation}
and
\begin{align}
    \label{eq:B_vs_prob2}
    \sum_{j=1}^d \int_0^t \frac{1}{K(u-)} \tilde F_j(\d u) &= P(T\leq t\given C\geq t) \nonumber\\
    &\phantom{{}=} + \frac{1}{K(t-)} \int_0^{t-}\Big(1- \P(\tilde T = s, \tilde D = 0 \given C = s) \\
    &\phantom{{}=} - \sum_{j=1}^d \int_0^s \frac{1}{K(u-)} \tilde F_j(\d u) \Big) G(\d s) \nonumber
\end{align}
hold for all $t \in \J$.
\end{lemma}

\begin{proof}
We know that $\sum_{j=1}^d F_j(t-) = 1 - S(t-)$ and similarly that $\tilde S(t-) + \sum_{j=0}^d \tilde F_j(t-) = 1$. Since $\tilde F_j(s) = \int_0^s a_j(u) F_j(\d u)$, we have $1-\sum_{j=1}^d \int_0^{t-} a_j(s) F_j(\d s) = \tilde S(t-) + \tilde F_0(t-)$. Recall that $B(s) = \int_0^{s-} S(u)^{-1} \tilde F_0(\d u)$. A change in the order of integration reveals that $\tilde F_0(t-) - \sum_{j=1}^d \int_0^{t-} B(s)  F_j(\d s) = \int_0^{t-} \P(T \geq t \given T > u) \tilde F_0(\d u) = S(t-) B(t)$, where the definition of $B$ is used once more.
Put together, this establishes that the equality
\begin{equation*}
    \sum_{j=1}^d \int_0^{t-} (1 - a_j(s) - B(s)) F_j(\d s) = \tilde S(t-) - S(t-) + S(t-) B(t)
\end{equation*}
holds for all $t \in \J$.
Now the desired result follows since $\P(\tilde T < t \given T \geq t) = (S(t-) - \tilde S(t-))/S(t-)$.
As for the second equality, similar arguments lead to $\sum_{j=1}^d \int_0^{t-} K(u-)^{-1} \tilde F_j(\d u) = (1- \tilde S(t-)K(t-)^{-1}) + K(t-)^{-1} \int_0^{t-} (1- \P(\tilde T = s, \tilde D = 0 \given C = s) - \sum_{j=1}^d \int_0^s K(u-)^{-1} \tilde F_j(\d u) ) G(\d s)$ and \eqref{eq:B_vs_prob2} then follows since $\P(T \leq t \given C \geq t) = (1- \check S(t) K(t-)^{-1}) = (1- \tilde S(t-) K(t-)^{-1}) + \sum_{j=1}^d \Delta \tilde F_j(t) / K(t-)$.
\end{proof}

\begin{lemma}
The equality 
\begin{equation} \label{eq:general_vs_cond_fixed}
  \begin{aligned}
    & \phantom{{}=} \P(T \leq t, D = j \given T > s) - \P(T \leq t, D = j \given \tilde T > s) \\
    &= \int_s^t F_j(t \given u) \frac{\tilde S(u-)}{\tilde S(s)} (\tilde H - H)(\d u)  + \int_s^t \frac{\tilde S(u-)}{\tilde S(s)} (H_j - \tilde H_j)(\d u) \\
    & \phantom{{}=} + \frac{1}{\tilde S(s)} \int_s^t \big ( F_j(t \given u) - \P(T \leq t, D = j \given \tilde T = u, \tilde D = 0) \big ) \tilde F_0(\d u)
\end{aligned}  
\end{equation}
holds for all $t\geq 0$ and $s\in\J$ with $s<t$.
\end{lemma}

\begin{proof}
As a preliminary step, we have $\P(T \leq t, D = j \given \tilde T > s) - \P(\tilde T \leq t, \tilde D = j \given T > s) = \P(T \leq t, D = j, \tilde T \leq t, \tilde D = 0 \given \tilde T > s) = \P(\tilde T > s)^{-1} \int_s^t \P(T \leq t, D = j \given \tilde T = u, \tilde D = 0) \tilde F_0(\d u)$.
On the other hand, an application of the Duhamel equation in $d+2$ dimensions, or a direct calculation, reveals that $\P(T \leq t, D = j \given T > s) - \P(\tilde T \leq t, \tilde D = j \given \tilde T > s) = \int_s^t \P(T \leq t, D = j \given T > u) \tilde S(u-) \tilde S(s)^{-1} (\tilde H_0 + \tilde H - H)(\d u) + \int_s^t \tilde S(u-) \tilde S(s)^{-1} (H_j - \tilde H_j)(\d u)$. Subtract the first expression from the second expression to obtain \eqref{eq:general_vs_cond_fixed}.
\end{proof}
\begin{lemma}
The equalities
\begin{equation} \label{eq:cond_fixed_vs_constant_sum}
\begin{aligned}
    &\phantom{{}=}\int_0^t \big(F_j(t \given s) - \P(T \leq t, D = j \given \tilde T = s, \tilde D = 0) \big) \tilde F_0(\d s) \\ &= \int_0^t (a_j(s) + B(s) - 1) F_j(\d s)
\end{aligned}
\end{equation}
and
\begin{equation} \label{eq:cond_fixed_vs_constant_sum2}
 \begin{aligned}
    &\phantom{{}=}\sum_{j=1}^d \int_0^t \big(\P(C \leq t \given C \geq s) - \P(C \leq t \given \tilde T = s, \tilde D = j) \big) \tilde F_j(\d s) \\ &= \int_0^t \big( \P(\tilde T = s, \tilde D = 0 \given C = s) + \sum_{j=1}^d \int_0^t\frac{1}{K(u-)} \tilde F_j(\d u) - 1 \big) G(\d s)
\end{aligned}  
\end{equation}
hold for all $t \in \J$ for alle $j=1, \dots, d$.
\end{lemma}
\begin{proof}
A change in the order of integration shows that $\int_0^t B(u) F_j(\d u) = \int_0^t F_j(t \given s) \tilde F_0(\d s)$. Split up $F_j(t) = \tilde F_j(t) + \int_0^t \P(T \leq t, D = j \given \tilde T = s, \tilde D = 0) \tilde F_0(\d s)$, where we have $\tilde F_j(t) = \int_0^t a_j(s) \d F_j(s)$, and put together to obtain \eqref{eq:cond_fixed_vs_constant_sum}.
The argument for \eqref{eq:cond_fixed_vs_constant_sum2} is similar.
\end{proof}

\section*{Appendix 2}
\subsection*{Convergence of the Aalen--Johansen estimator}

For an i.i.d.\ sample $(\tilde T_1,\tilde D_1),\ldots,(\tilde T_n,\tilde D_n)$ of $(\tilde T,\tilde D)$, we let $\hat H_{j,n}$ denote the Nelson--Aalen estimator for $H_j$ and $\hat H_n=\sum_{j=1}^d\hat H_{j,n}$. If we define the $(d+1)\times (d+1)$ matrix
\begin{align*}
    \hat{\mat{H}}_n(t)=\begin{pmatrix}-\hat H_n(t) & \hat H_{1,n}(t) & \cdots & \hat H_{d,n}(t)\\
    0 & 0 & \cdots & 0 \\
    \vdots & \vdots & \ddots & \vdots \\
    0 & 0 & \cdots & 0\end{pmatrix}
\end{align*}
then the Aalen--Johansen estimator is defined as
\begin{align}
    \label{eq:def_aj}
    \hat{\mat{P}}_n(t)=\Prodi_0^t(\identmat+\hat{\mat{H}}_n(\d s)).
\end{align}
By arguments similar to Section 4.2 of \cite{Gill1990}, we see that $\sup_{s \in [0,t]}|\hat{H}_{j,n}(s)-\tilde H_j(s)|\to 0$ almost surely for $n\to\infty$ for all $j$ and $t\in\J$ and thus also $\sup_{s \in [0,t]}\|\hat{\mat{H}}_n(s)-\tilde{\mat{H}}(s)\| \to 0$ almost surely for $n\to\infty$ for all $t\in\J$. By continuity of the product integral \citep{Gill1990} we conclude that
\begin{align}
    \label{eq:limit_aj}
    \sup_{s \in [0,t]}\|\hat{\mat{P}}_n(s)-\Prodi_0^s(\identmat+\tilde{\mat{H}}(\d u))\| \to 0
\end{align}
almost surely as $n\to\infty$ for all $t\in\J$.

The Kaplan--Meier estimator for the all-cause survival function $S$ is defined as
\begin{align}
    \label{eq:kaplanmeier}
    \hat{S}_n(t)=\Prodi_0^t(1-\hat H_n(\d s))
\end{align}
which is just entry $(1,1)$ of $\hat{\mat{P}}_n(t)$.

\section*{Appendix 3}
\subsection*{Constructing latent times}

Let us consider a probability space $(\Omega, \F, \P)$ on which random variables $\tilde T \from \Omega \to (0, \infty)$ and $\tilde D \from \Omega \to \{0, \dots, d\}$ are defined. 
We now want to extend the probability space in order to define a random variable $C$ that satisfies $C \geq \tilde T$ when $\tilde D \neq 0$ and $C = \tilde T$ when $\tilde D = 0$ and follows a certain conditional distribution given $(\tilde T, \tilde D)$. 
The desired conditional cumulative distribution function is given by $F_{C \given \tilde T, \tilde D}(c \given \tilde t, \tilde d) = \indic{\tilde t \leq c} \indic{\tilde d = 0} + (1- \prodi_{\tilde t-}^c(1 - \check H_0(\d u))) \indic{\tilde t \leq c} \indic{\tilde d \neq 0}$ where $\check H_0$ is defined in \eqref{eq:check_H0} solely based on the distribution of $(\tilde T, \tilde D)$.
For given $\tilde t$ and $\tilde d$, the function $c \mapsto F_{C \given \tilde T, \tilde D}(c \given \tilde t, \tilde d)$ is right-continuous and increasing.
For a right-continuous and increasing function $f \from \R \to \R$, the inversion, as defined in for instance Section~II.2a of \cite{asmussen2007stochastic}, given by $f^{\leftarrow}(u) = \inf\{x : f(x) \geq u\} \in \R \cup \{- \infty, \infty\}$ is a useful concept. Using right continuity and that $f$ is increasing, the conclusion that $u \leq f(x)$ if and only if $f^{\leftarrow}(u) \leq x$ can be reached.
Extend the sample space to $\Omega' = \Omega \times [0,1]$ and the $\sigma$-algebra to $\F' = \F \times \sigAlg{B}([0,1])$, where $\sigAlg{B}$ is the Borel $\sigma$-algebra, and the probability measure by $\P'(A \times (u,v]) = \P(A) (v-u)$. By these extensions, the random variable $U$ defined on the probability space $(\Omega', \F', \P')$ and given by $U(\omega') = u$ for $\omega'= (\omega, u)\in\Omega'$ follows a uniform distribution on $[0,1]$ and is independent of $(\tilde T, \tilde D)$. 
The random variable defined by $C(\omega') = F_{C \given \tilde T, \tilde D}^{\leftarrow}(u \given \tilde T(\omega), \tilde D(\omega))$ for $\omega' = (\omega, u)$, where $u \mapsto F_{C \given \tilde T, \tilde D}^{\leftarrow}(u \given \tilde t, \tilde d)$ is the inversion of $c \mapsto F_{C \given \tilde T, \tilde D}(c \given \tilde t, \tilde d)$, now fulfills $C(\omega') \geq \tilde T(\omega)$ when $\tilde D(\omega) \neq 0$ and $C(\omega') = \tilde T(\omega)$ when $\tilde D(\omega) = 0$ and has the desired conditional distribution.
Renaming $\P'$ to $\P$, we note that, for $t \leq s$,
\begin{equation} \label{eq:constr_C_vs_TD}
    \begin{aligned}
    \P(\tilde T \leq t, \tilde D = j, C > s) &= \int_0^t \P(C > s \given \tilde T = u, \tilde D = j) \tilde F_j(\d u) \\
    &= \int_0^t \Prodi_{u-}^s(1 - \check H_0(\d u)) \tilde F_j(\d u) \\
    &= \Prodi_0^s(1 - \check H_0(\d u)) \int_0^t \Prodi_0^{u-}(1 - \tilde H(\d v)) \d \tilde H_j(u),
    \end{aligned}
\end{equation}
where the last equation uses $\tilde S(s) = \prodi_0^s(1-\check H_0(\d u)) \prodi_0^s(1-\tilde H(\d u))$ and $\tilde H_j(s) = \int_0^s \tilde S(u-)^{-1} \tilde F_j(\d u)$. Since $\prodi_0^s(1-\tilde H(\d u)) + \sum_{j=1}^d \int_0^s \prodi_0^{u-} (1 - \tilde H(\d v)) \tilde H_j(\d u) = 1$, these equalities can also be used to establish that
\begin{equation} \label{eq:constr_C_dist}
\begin{aligned}
    \P(C > s) &= \tilde S(s) + \sum_{j=1}^d\P(\tilde T \leq s, \tilde D = j, C > s) \\
    &= \Prodi_0^s (1 - \check H_0(\d u)).
\end{aligned}
\end{equation}
This reveals that the constructed $C$ is proper when and only when $\prodi_0^s(1- \check H_0(\d u)) \to 0$ for $s \to \infty$. 

In a setting identical to above, we now want to construct $(T,D)$ such that $(T, D) = (\tilde T, \tilde D)$ when $\tilde D \neq 0$ and $T > \tilde T$ when $\tilde D = 0$ and such that $(T,D)$ has a certain conditional distribution given $(\tilde T, \tilde D)$. Here, the desired conditional cumulative distribution function is given by $F_{T,D \given \tilde T, \tilde D}(t,j \given \tilde t, \tilde d) = \indic{\tilde t \leq t} \indic{0 < \tilde d \leq j} + \sum_{k=1}^{j}\int_{\tilde t}^t \prodi_{\tilde t}^{s-}(1-\tilde H(\d u)) \tilde H_k(\d u) \indic{\tilde d = 0} \indic{\tilde t \leq t}$.
This can be achieved in a manner similar to above. A two step procedure is to construct $D$ according to the conditional cumulative distribution function given by $F_{D \given \tilde T, \tilde D}(j \given \tilde t, \tilde d) = F_{T,D \given \tilde T, \tilde D}(\infty,j \given \tilde t, \tilde d) (F_{T,D \given \tilde T, \tilde D}(\infty,d \given \tilde t, \tilde d))^{-1}$ and then to construct $T$ according to the conditional cumulative distribution function given by 
\begin{equation*}
    F_{T \given D, \tilde T, \tilde D}(t \given j, \tilde t, \tilde d) = \frac{F_{T,D \given \tilde T, \tilde D}(t,j \given \tilde t, \tilde d) - F_{T,D \given \tilde T, \tilde D}(t,j - 1 \given \tilde t, \tilde d)}{F_{D \given \tilde T, \tilde D}(j \given \tilde t, \tilde d) - F_{D \given \tilde T, \tilde D}(j-1 \given \tilde t, \tilde d)}
\end{equation*}
where division by 0 can be taken to produce 0.
The so constructed pair $(T, D)$ satisfies in particular $\P(T \leq t, D = j \given \tilde T = s, \tilde D = 0) = \int_s^{t} \prodi_s^{u-}(1- \tilde H(\d v)) \tilde H_j(\d u)$ for $t \geq \tilde t$ and $s \in \J$ for $j=1, \dots, d$ and so also
\begin{equation*}
\begin{aligned}
    \P(T \leq t, D = j) &= \tilde F_j(t) + \int_0^t \P(T \leq t, D = j \given \tilde T = s, \tilde D = 0) \tilde F_0(\d s) \\
    &= \int_0^t \tilde S(u-) \tilde H_j(\d u) + \int_0^t \int_s^t \Prodi_s^{u-}(1- \tilde H(\d v)) \tilde H_j(\d u) \tilde F_0(\d s) \\
    &=\int_0^t\big( \tilde S(u-) +\int_0^{u-} \Prodi_s^{u-}(1-\tilde H(\d v)) \tilde F_0(\d s) \big) \tilde H_j(\d u) \\
    &= \int_0^t \Prodi_0^{u-}(1-\tilde H(\d v)) \big( \Prodi_0^{u-}(1- \check H_0(\d v)) + \int_0^{u-} \Prodi_0^{s-}(1-\check H_0(\d v)) \check H_0(\d s) \big) \tilde H_j(\d u) \\
    &= \int_0^t \Prodi_0^{u-}(1-\tilde H(\d v)) \tilde H_j(\d u).
\end{aligned}
\end{equation*}

These constructions lead to the following proposition, which has similarities to Theorem~2 of \cite{tsiatis1975nonidentifiability}.
\begin{proposition}\label{prop:existence}
Given positive random variable $\tilde T$ and random variable $\tilde D$ with values in $\{0, \dots, d\}$, positive random variables $T$ and $C$ and random variable $D$ with values in $\{1,\ldots,d\}$ exist such that $\tilde T=T\wedge C$, $\tilde D=D \indic{T \leq C}$ and such that $(T, D)$ is independent of $C$.
\end{proposition}
\begin{proof}
Construct $C$, $T$ and $D$ as described above. The distribution of $C$ is given by $\P(C > t) = \prodi_0^t(1- \check H_0(\d s))$ and the distribution of $(T, D)$ is given by $\P(T \leq t, D=j) = \int_0^t \prodi_0^{s-}(1-\tilde H(\d u)) \tilde H_j(\d s)$. Equation~\eqref{eq:constr_C_vs_TD} reveals that, for $t \leq s$, $\P(T \leq t, D = j, C > s) = \P(\tilde T \leq t, \tilde D = j, C > s) = \P(C > s) \P(T \leq t, D = j)$.
By construction we have, for $t > s$, $\P(T \leq t, D = j \given \tilde T = s, \tilde D = 0) = \P(T \leq t, D = j \given T > s)$. According to Proposition~\ref{prop:strong_censoring}, this implies $\P(T \leq t, D = j \given \tilde T > s) = \P(T \leq t, D = j \given T > s)$ for all $t \geq 0$ and $s \in \J$. 
Using this and \eqref{eq:productstruc_stilde}, we have, for $t > s$ with $s \in \J$,
\begin{equation*}
\begin{aligned}
    \P(T \in (s,t], D = j, C > s) &= \P(T \leq t, D = j \given \tilde T > s) \P(\tilde T > s) \\
    &= \int_s^t \Prodi_s^{u-}(1 - \tilde H(\d v)) \tilde H_j(\d u) \Prodi_0^s(1- \tilde H(\d u)) \Prodi_0^s(1- \check H_0(\d u)) \\
    &= \int_s^t \Prodi_0^{u-}(1 - \tilde H(\d v)) \tilde H_j(\d u)  \Prodi_0^s(1- \check H_0(\d u)) \\
    &= \P(T \in (s,t], D = j) \P(C > s).
\end{aligned}
\end{equation*}
Both sides are 0 if $s \in (0, \infty)\backslash \J$. We have thereby established that $\P(T \leq t, D = j, C > s) = \P(T \leq t, D = j) \P(C > s)$ for all $t, s \geq 0$ and $j = 1, \dots, d$, and so $(T,D)$ and $C$ are independent.
\end{proof}

\end{document}